\documentclass[12pt,reqno]{article}
\usepackage{color,amsmath,amssymb,lscape,graphicx,ifthen,pdfpages}
\usepackage{color}
\usepackage{fullpage}
\usepackage{float}

\usepackage{psfig}
\usepackage{amsthm}
\usepackage{amsfonts}
\usepackage{latexsym}
\usepackage{epsf}

\newcommand{\seqnum}[1]{\href{http://oeis.org/#1}{\underline{#1}}}

\usepackage[colorlinks=true,
linkcolor=webgreen,
filecolor=webbrown,
citecolor=webgreen]{hyperref}
\definecolor{webgreen}{rgb}{0,.5,0}\definecolor{webbrown}{rgb}{.6,0,0}
\begin{document}

%\begin{center}
%\epsfxsize=4in
%\leavevmode\epsffile{logo129.eps}
%\end{center}

\theoremstyle{plain}
\newtheorem{theorem}{Theorem}
\newtheorem{corollary}[theorem]{Corollary}
\newtheorem{lemma}[theorem]{Lemma}
\newtheorem{proposition}[theorem]{Proposition}

\theoremstyle{definition}
\newtheorem{definition}[theorem]{Definition}
\newtheorem{example}[theorem]{Example}
\newtheorem{conjecture}[theorem]{Conjecture}

\theoremstyle{remark}
\newtheorem{remark}[theorem]{Remark}

\newtheorem{case}{Case}

%%%%%%%%%%%%%%%%%%%%%%%%%%%%%%%
\def\glqq{,\,\!\!,}
\def\grqq{`\,\!`}
\def\eg{{\it e.g.},\,}
\def\Eg{{\it E.g.},\,}
\def\ie{{\it i.e.},\,}
\def\viz{{\it viz}\ }
\def\etc{\, {\it etc.}\,}
\def\sspp{\,+\,}
\def\sspm{\,-\,}
\def\sspeq{\,=\,}
\def\sspdef{\, :=\,}
\def\sspkl{\,<\,}
\def\sspgr{\,>\,}
\def\sspgeq{\,\geq\, }
\def\sspleq{\,\leq\,} 
\def\spgeq{\,\geq\, }
\def\spleq{\,\leq\,} 
\def\pn{\par\noindent}
\def\pbn{\par\bigskip\noindent}
\def\psn{\par\smallskip\noindent}
\def\sspentsp{\, \hat =\, }
\def\sspequiv{\,\equiv\,}
\def\sspnotequiv{\,\not\equiv  \,}
\def\sspmapsto{\,\mapsto\,}
\def\sspdiv{\,|\,}
\def\Beq{\begin{equation}}
\def\Eeq{\end{equation}}
\def\Bequarray{\begin{eqnarray}}
\def\Eequarray{\end{eqnarray}}
\def\sspin{\,\in\,}
\def\sspfed{\,=:\,}
\def\sspto{\,\to\,}
\def\dstyle#1{$\displaystyle #1 $}
\def\union{\cup}
\def\floor#1{\left\lfloor{#1}\right\rfloor}
\def\ceil#1{\left\lceil{#1}\right\rceil}
\def\abs#1{\vert {\,#1\,} \vert}
\def\Caseszwei#1#2#3#4{\left\{ \begin{array}{ll}#1&\mbox{#2}\\ &\\#3&\mbox{#4}\end{array}\right.}
\def\binomial#1#2{{#1} \choose {#2}}
\def\sspdiv{\,|\,} 
\def\rprod{\prod\llap {\raise 8pt\hbox{$\rightarrow \thinspace$}}} 
%%%%%%%%%%%%%%%%%%%%%%%%%%%%%%%%
\rightline{Karlsruhe} \par\smallskip\noindent
\rightline{April 09  2014}
\vbox {\vspace{6mm}}
\begin{center}
\vskip 1cm{\LARGE\bf 
On Collatz' Words, Sequences and Trees
}
\vskip 1cm
\large
Wolfdieter L a n g \footnote{ \url{http://www.itp.kit.edu/~wl}} \\
Karlsruhe \\
Germany\\
\href{mailto:wolfdieter.lang@partner.kit.edu}{\tt wolfdieter.lang@partner.kit.edu}
\end{center}

\vskip .2 in
\begin{abstract}
Motivated by a recent work of  Tr\"umper \cite{Truemper} we consider the general Collatz word (up-down pattern) and the sequences following this pattern. The recurrences for the first and last sequence entries are given, obtained from repeated application of the general solution of a binary linear inhomogeneous Diophantine equation. These recurrences are then solved. The Collatz tree is also discussed.    
\end{abstract}

\bigskip

\section{Introduction}
The Collatz map $C$ for natural numbers maps an odd number $m$ to $3\,m\sspp 1$ and an even number to \dstyle{\frac{m}{2}}.  The {\sl Collatz} conjecture  \cite{Lagarias}, \cite{WikiC}, \cite{WeissteinC} is the claim that every natural number $n$ ends up, after sufficient iterations of the map $C$, in the trivial cycle $(4,\, 2,\, 1)$. Motivated by the work of {\sl Tr\"umper} ~\cite{Truemper} we consider a general finite {\sl Collatz} word on the alphabet $\{u,\,d\}$, where $u$ (for `up') indicates application of the map $C$ on an odd number, and $d$ (for `down')  for applying the map $C$ on an even number. The task is to find all sequences which follow this word pattern (to be read from the left to the right). These sequences will be called $CS$ (for {\sl Collatz} sequence also for the plural) realizing the CW (for {\sl Collatz} word also for the plural) under consideration. This problem was solved by {\sl Tr\"umper} \cite{Truemper} under the restriction that the first and last sequence entries are odd. Here we shall not use this restriction.The solution will be given in terms of recurrence relations for the first and last entries of the $CS$ for a given $CW$. This involves a repeated application of the general solution on positive numbers of the linear inhomogeneous {\sl Diophant}ine equation $a\,x\sspp b\,y\sspeq c$, with $a\sspeq 3^m$ and $b\sspeq 2^n$ and given integer $c$. Because $gcd(3,2)\sspeq 1$ one will always have a countable infinite number of solutions. This general solution depends on a non-negative integer parameter $k$. We believe that our solution is more straightforward than the one given in \cite{Truemper}.      
%%%%%%%%%%%%% end Introduction  %%%%%%%%%%%%%%%%%%%%
\section{Collatz words, sequences and the Collatz tree}\label{section2}
The {\sl Collatz} map $C:\ \mathbb{N} \to \mathbb{N},\ m \mapsto 3\,m\sspp 1$ if $m$ is odd,  \dstyle{m \mapsto \frac{m}{2}} if $m$ is even, leads to an increase $u$ (for `up') or decrease $d$ (for `down'), respectively. Finite {\sl Collatz} words over the alphabet $\{u,\,d\}$ are considered with the restriction that, except for the one letter word $u$, every $u$ is followed by a $d$, because $2\,m\sspp 1 \sspmapsto 2\,(3\,m+1)$. This is the reason for introducing (with \cite{Truemper}) also $s\sspdef ud$. Thus $s$ stands for  $2\,m\sspp 1\sspmapsto 3\,m\sspp 1$. The general finite word is encoded by an $(S+1)-$tuple $\vec n_{S}\sspeq [n_0,\,n_1,\,...,\,n_S ]$ with $S\sspin \mathbb N$. 
\Bequarray \label{CW}
 CW(\vec n_{S+1}) &\sspeq& d^{n_0}\,s\,d^{n_1-1}\,s\, \cdots\, s\,d ^{n_S-1} \\
                              &\sspeq &   (d^{n_0}\,s)\,(d^{n_1-1}\,s)\, \cdots\, \,(d ^{n_{S-1}-1}\,s)\, d ^{n_S-1}\ , \nonumber
\Eequarray
with $n_0\sspin {\mathbb N}_0\sspdef {\mathbb N}\, \union\, \{0\} $, $n_i\sspin \mathbb N$, for $i\sspeq 1,\,2,\,...,\, S$. The number of $u$ (that is  of $s\sspeq ud$) letters in the word $CW(\vec n_{S})$, or $CW(S)$, for short, is $S$ (which is why we have used $\vec n_S$ not $\vec n_{S+1}$ for the $S+1$ tuple), and the number of $d$ is \dstyle{D(S)\sspdef \sum_{j=0}^{S}\,n_j }. In \cite{Truemper} $n_0\sspeq 0$ (start with an odd number), $y\sspeq S$ and $x\sspeq D(S)$.\psn
Some special words are not covered by this notation: first the one letter word $u$ with the {\sl Collatz} sequence ($CS$) of length two $CS(u;\,k)\sspeq [2\,k\sspp 1,\, 2\,(3\,k\sspp  2)]$, and $CW([n_0])\sspeq d^{n_0}$ with the family of  sequences $CS([n_0];k) \sspeq [2^{n_0}\,k,2^{n_0 - 1}\,k,\,...,\,1\,k] $ with  $k\sspin {\mathbb N}_0$. \psn
A {\sl Collatz} sequence $CS$ realizing a word $CW(\vec n_{S})$  is of length $L\sspeq D\sspp S\sspp 1$ and follows the word pattern from the left to the right. $CS(\vec n_{S})\sspeq [c_1,\, c_2,\,...,\,c_L]$. For example, $CW([1,2,1])\sspeq dsds$ with  $S\sspeq 2$,  $D(2)\sspeq 2$, and length $L\sspeq 7$ with $SC_0(\vec n_S) \sspeq [2, 1, 4, 2, 1, 4, 2]$ is the first of these sequences (for non-negative integers), the one with smallest start number $c_1$. In order to conform with the notation used in \cite{Truemper} we shall use for the start number $c_1\sspeq M$ and for the last number $c_L \sspeq N$. However, in \cite{Truemper} $M$ and $N$ are restricted  to be odd which will not be the case here. Later one can get the words with odd start number M by choosing $n_0\sspeq 0$. In order to have also $N$ odd one has to pick from $SC_k([0,n_1,...,n_S])$, for $k\sspin \mathbb N_0$ only the odd members.\psn
In \cite{Truemper} the monoid of Collatz words, with the unit element $e\sspeq$ {\it empty word} is treated. This will not be considered in this work. Also the connection to the $3\,m\sspm 1$ problem will not be pursued here.\psn
The {\sl Collatz} tree CT is an infinite (incomplete) ternary tree, starting with the root, the number $8$ on top at level $l=0$. Three  branches, labeled $L$, $V$ and $R$ can be present:  
If a node (vertex) has label $n\sspequiv 4\,(mod\,6))$ the out-degree is $2$ with the a left edge (branch) labeled $L$ ending in a node with label \dstyle{\frac{n-1}{3}} and a right edge (label $R$) ending in the node labeled $2\,n$. In the other cases, $n\sspequiv 0,\,1,\,2,\,3,\,5,\,(mod\,6)$, with out-degree $1$, a vertical edge (label $V$) ends in the node labeled $2\,n$. The root labeled $8$ stands  for the trivial cycle $8$\, repeat$(4,\,2,\,1)$. See the {Figure 1} for $CT_7$ with only the first eight levels. It may seem that this tree is left-right  symmetric (disregarding the node labels) but this is no longer the case starting at level $l\sspeq 12$. At level $l=10$ the $mod\, 6$ structure of the left and right part of $CT$, also taking into account the node labels, is broken for the first time, but the node labels $4 (mod\, 6)$ are still symmetric. At the next level $l\sspeq 11$ the left-right symmetry concerning the labels $4 (mod\, 6)$ is also broken, leading at level $l\sspeq 12$ to a symmetry breaking in the branch structure of the left and right part of $CT$. Thus at level $l\sspeq 12$ the number of nodes becomes odd for the first time: $15$ nodes on the left side versus $14$  nodes on the right one. See rows $ l \sspp  3$ of \seqnum{A127824} for the node labels of the first levels, and \seqnum{A005186}$(l+3)$ for the number of nodes. The number of $4\, (mod\, 6)$ nodes at level $l$ is given in \seqnum{A176866}$(l+4)$. \psn
A $CS$ is determined uniquely from its start number $M$. Therefore no number can appear twice in $CT$, except for the numbers $1,\,2,\,4$ of the (hidden) trivial cycle. The Collatz conjecture is that every natural number appears in CT at some level ($1, 2,$ and $4$ are hidden in the root $8$). A formula for $l\sspeq l(n)$ would prove the conjecture.
\psn
Reading $CT$ from bottom to top, beginning with some number $M$ at a certain level $l$, recording the edge labels up to level $l=0$, leads to a certain $L,V,R$-sequence. E.g., $M\sspeq 40$ at level $l\sspeq 5$ generates the length $5$ sequence $[V,R,V,L,V]$. This is related to the CS starting with $M\sspeq 40$, namely  $[40,20,10,5,16,8]$, one of the realizations of the CW  $d,d,d,u,d\sspeq d^3s$, with $S\sspeq 1$ and $\vec n_1\sspeq [3,1]$. (Later it will be seen that this is the realizations with the third smallest start number, the smaller once are $8$ and $24$). One has to map $V$ and $R$ to $d$ and $L$ to $u$. This shows that the map from a $L,V,R$-sequence to a CW is not one to one. The numbers $n\sspequiv 4\, (mod\ 6)$ except $4$ (see \seqnum{A016957}) appear exactly in two distinct CS. For example, $64 \sspequiv 4\,(mod\, 6)$ shows up in all $CS$ starting at any vertex which descends from the bifurcation at $64$, \eg $21,\,128$; $42,\,256$; $84,\,85,\,512$; \etc  \pbn

%%%%%%%%%%%%%%%% Figure %%%%%%%%%%%%%%%%%%%%%
\psn
\parbox{16cm}{\begin{center}
{\includegraphics[height=10cm,width=.8\linewidth]{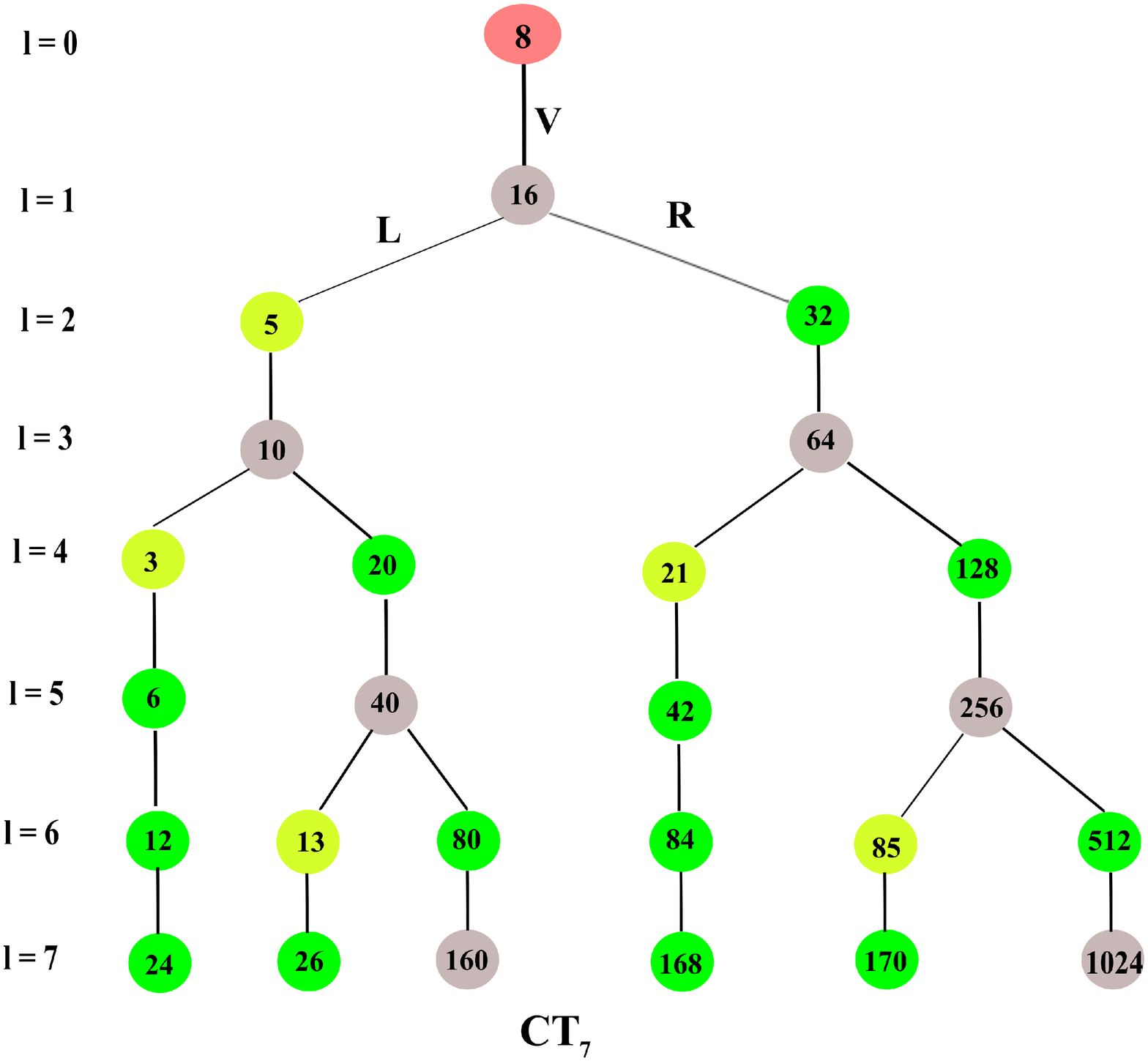}}
\end{center}
}
\psn
\hskip 6cm {\bf  Figure: Collatz Tree $\bf CT_7$}

%%%%%%%%%%%%%%%%%%%%%%%%%%%%%%%%%%%%%%
\section{Solution of a certain linear inhomogeneous Diophantine equation}\label{section3}
The derivation of the recurrence relations for the start and end numbers $M$ and $N$ of {\sl Collatz} sequences ($CS$) with prescribed up-down pattern (realizing  a given  $CW$) we shall need the general solution of the following linear and inhomogeneous {\sl Diophant}ine equation.  
\Beq\label{Diophant}
D(m,n;c):\hskip 2cm 3^m\,x\sspm 2^n\,y\sspeq c(m,n),\ \ m\sspin \mathbb N_0, \ n\sspin \mathbb N_0,\ c(m,n) \sspin \mathbb Z\ .
\Eeq
It is well known \cite{NZM}, pp. 212-214, how to solve the equation $a\,x\sspp b\,y\sspeq c$
for integers $a,\, b$ (not $0$) and $c$ provided $g\sspeq gcd(a,\,b)$ divides $c$ (otherwise there is no solution) for integers $x$ and $y$. One will find a sequence of solutions parameterized by $t\sspin \mathbb Z$. Then one has to restrict the $t$ range to obtain all positive solutions. The procedure is to find first a special solution $(x_0,\, y_0)$ of the equation with $c\sspeq g$. Then the general solution is $(x\sspeq \frac{c}{g}\, x_0\sspp \frac{b}{g}\, t,\, y\sspeq \frac{c}{g}\, y_0\sspm \frac{a}{g}\, t)$ with $t\sspin Z$. The proof is found in \cite{NZM}. For our problem $g\sspeq gcd(3^m,\, 2^n)\sspeq 1$ for non-negative $m,\, n$ which will divide any $c(m,\, n)$. \pbn
\begin{lemma}\label{SolDiophant} {\bf Solution of $\bf D(m,n;c)$}\psn    
{\bf a)} A special positive integer solution of  D(m,n;1) is\psn
\Bequarray\label{x0y0}
y_0(m,n) &\sspeq& \left(\frac{3^m\sspp 1}{2}\right)^{n\sspp 3^{m-1}}\, (mod\, 3^m)\ , \nonumber\\
x_0(m,n) &\sspeq& \frac{1\sspp 2^n\,y_0(m,n)}{3^m}\ .
 \Eequarray
 {\bf b)} The general solution with positive $x$ and $y$ is\psn
\Bequarray\label{Solxy}
x(m,n) &\sspeq& c(m,n)\,x_0(m,n)\sspp 2^n\,t_{min}(m,n;sign(c)) \sspp 2^n\,k \ ,\nonumber  \\
y_0(m,n) &\sspeq& c(m,n)\,y_0(m,n)\sspp 3^m\,t_{min}(m,n;sign(c)) \sspp 3^m\,k \ ,
 \Eequarray
with $k\sspin \mathbb N_0$, and 
\Beq\label{tmin}
t_{min}(m,n;sign(c))\sspeq {\Caseszwei{\ceil{\abs{c(m,n)}\,\frac{x_0(m,n)}{2^n}}}{${\rm if} \ c\sspkl 0 \ ,$} {\ceil{-c(m,n)\,\frac{y_0(m,n)}{3^m}}} {${\rm if} \ c\sspgeq 0\ .$}}
\Eeq
\end{lemma}
\noindent
For the proof we shall use the following {\sl Lemma}:
\begin{lemma}\label{Anm}    
$ A(n,m)\sspdef {\binomial{n-1}{m-1}}\, \frac{gcd(m,n)}{m}$\, is a positive integer for $ m\sspeq 1,\,2\, ...\, n,\  n\sspin \mathbb N$\, .
\end{lemma}
\begin{proof}\psn
[due to {\sl Peter Bala}, see \seqnum{A107711}, history, Feb 28 2014]:\pn 
This is the triangle \seqnum{A107711} with A(0,0) =1. 
By a rearrangement of factors one also has  \dstyle{A(n,m)\sspeq {\binomial{n}{m}}\,\frac{gcd(n,m)}{n}}. Use $gcd(n,m)\,lcm(m,n)\sspeq n\, m$ (\eg \cite{FR}), theorem 2.2.2., pp. 15-16, where also the uniqueness of the $lcm$ is shown). \dstyle{A(n,m)\sspeq \frac{a(n,m)}{lcm(n,m)}} with \dstyle{a(n,m)\sspeq {\binomial{n}{m}}\,m}, a positive integer because the binomial is a combinatorial number. $m\sspdiv a(n,m)$ and $n\sspdiv a(n,m)$ because \dstyle{a(n,m) \sspeq n\,{\binomial{n-1}{m-1}}} by a rearrangement. Hence $a(n,m)\sspeq k_1\,m\sspeq k_2\,n$, \ie $a(n,m)$ is a common multiple of $n$ and $m$ (call it $cm(n,m)$). $lcm(n,m)\sspdiv a(n,m)$ because $lcn(n,m)$ is the (unique) lowest $cm(n,m)$. Therefore \dstyle{\frac{a(n,m)}{lcm(n,m)}\sspin \mathbb N}, since only natural numbers are in the game.   
\end{proof}\noindent
Now to the proof of  {\sl Lemma\ 1}.\pn
\begin{proof} \psn
{\bf a)}  \dstyle{x_0(m,n) \sspeq \frac{1\sspp 2^n\,y_0(m,n)}{3^m}} is a solution of $D(m,n;1)$ for any $y_0(m,n)$. The given $y_0(m,n)$ is a positive integer $\sspin \{1,\,2,\,...,\, 3^{m}\sspm 1\},\  m\sspin \mathbb N$ and $y_0(0,n)\sspeq 1$ for $n\sspin \mathbb N_0$. One has to prove that $x_0(m,n)$ is a positive integer. This can be done by showing that $1\sspp 3^n\,y_0(n,m)\sspequiv 0\,(mod\,3^m)$ for $m\sspin \mathbb N$. One first observes that \dstyle{\frac{3^m\sspp 1}{2}\sspequiv \frac{1}{2}\,(mod\, 3^m)}, because obviously \dstyle{2\,\frac{3^m\sspp 1}{2} \sspequiv 1\, (mod\, 3^m)} ($2$ is a unit in the ring $\mathbb Z_{3^m}$). For $m\sspeq 0$ one has $x_0(0,n) \sspeq 1\sspp 2^n$, $n\sspin \mathbb N_0$, which is positive. In the following $m\sspin \mathbb N$.  
\Beq
1\sspp 2^n\,\left(\frac{3^m\sspp 1}{2}\right)^{n+3^{m-1}} \sspequiv 1 + 2^n\, \left(\frac{1}{2}\right)^{n+3^{m-1}} \sspequiv 1\sspp  \left(\frac{1}{2}\right)^{3^{m-1}}\, (mod\, 3^m)\ .
\Eeq   
Now we show that \dstyle{L(m)\sspdef  \left(\frac{3^m\sspp 1}{2}\right)^{3^{m-1}}\sspequiv 0\, (mod\, 3^m)} by using \dstyle{\frac{3^m\sspp 1}{2} \sspeq 3\, k(m)\sspm 1} with  \dstyle{k(m)\sspdef \frac{3^{m-1}\sspp1}{2}}, a positive integer. The binomial theorem leads with $ a(m)\sspeq 3^{m\sspm 1}$ to
\Bequarray
(3\,k(m)\sspm 1)^{a(m)} &\sspeq& \sum_{j=0}^{a(m)-1}\,{\binomial{a(m)}{j}}\, (-1)^j\, (3\,k(m))^{a(m)-j} \nonumber \\
&\sspeq& 3^m\,\Sigma_1(m) \sspp \Sigma_2(m),  {\rm with} \\
\Sigma_1(m) &\sspeq& \sum_{j=0}^{a(m)-m}\, (-1)^j\,{\binomial{a(m)}{j}}\,k(m)^{a(m)-j}\,3^{a(m)-m-j}, \ {\rm and} \\
\Sigma_2(m) &\sspeq& \sum_{j=a(m)-m+1}^{a(m)-1}\, (-1)^j\,{\binomial{a(m)}{j}}\,(3\, k(m))^{a(m)-j}\ . 
\Eequarray 
$\Sigma_1(m)$ is an integer because of $a(m)-j\sspgeq a(m)-m-j\sspgeq 0$  and the integer binomial, hence $L(m)\sspequiv \Sigma_2 \,(mod\, 3^m)$. Rewriting $\Sigma_2$ with $j'\sspeq j-a(m)+m-1$, using also the symmetry of the binomial, one has
\Bequarray 
\Sigma_2(m)&\sspeq& \sum_{j=0}^{m-2}\,{\binomial{a(m)}{j}}\, (-1)^{j-m}\,(3\,k(m))^{m-1-j}\nonumber \\
&\sspeq& \sum_{j=1}^{m-1}\,(-1)^{j+1}\, {\binomial{a(m)}{j}}\, (3\,k(m))^j \sspeq 3^m\,\widehat\Sigma_2(m)\, \ {\rm with}\\
\widehat\Sigma_2(m)&\sspeq& \sum_{j=1}^{m-1}\,(-1)^{1+j}\,{\binomial{a(m)}{j}}\,k(m)^j\,3^{j-m}\nonumber \\
&\sspeq&  \sum_{j=1}^{m-1}\,(-1)^{1+j}\,k(m)^j\,{\binomial{a(m)-1}{j-1}}\,\frac{1}{j}\,3^{j-1}\ .
\Eequarray
In the last step a rearrangement of the binomial has been applied, remembering that $a(m)\sspeq 3^{m-1}$. It remains to be shown that \dstyle{A_{m,j}\sspdef  3^{j-1}\,{\binomial{3^{m-1}-1}{j-1}}\,\frac{1}{j}} is a(positive) integer for $j\sspeq 1,\,2,\,...\,m-1$. Here {\sl Lemma 2}  comes to help. Consider there $A(3^{m-1},j)$ for $j\sspeq 1,\,2,\,...,\,m-1$  ($m=0$ has been treated separately above), which is a positive integer. If $3\, \not|\, j$ then $3^{j-1}\,A(3^{m-1},j) \sspeq A_{m,j}$, hence a positive integer. If $j\sspeq 3^k\,J$, with $k\sspin \mathbb N$ the largest power of $3$ dividing $j$ then $gcd(3,J)\sspeq 1$, and  $j\sspeq 3^k\,J \sspleq m-1\sspkl 3^{m-1}$ and $gcd(3^{m-1},3^k\,J)\sspeq 3^q$ with $q\sspeq min(k,m-1)$.
\end{proof}\noindent
\begin{proof} \psn
{\bf b)} The general integer solution of ~\ref{Diophant} is then (see \cite{NZM}, pp. 212-214; note that there $b>0$, here $b<0$, and we have changed $t\mapsto -t$) 
\Bequarray
x\sspeq \hat x(m,n;t)&\sspeq& c(m,n)\,x_0(m,n)\sspp 2^n\,t\, ,\nonumber \\
y\sspeq \hat y(m,n;t)&\sspeq& c(m,n)\,y_0(m,n)\sspp 3^m\,t\, , \ t\sspin \mathbb Z\,.
\Eequarray
In order to find all  positive solutions for $x$ and $y$ one has to restrict the $t$ range, depending on the sign of $c$. If $c(m,n)\sspgeq 0$  then, because $x_0$ and $y_0$ are positive and \dstyle{\frac{x_0(m,n)}{2^n}\sspeq \frac{y_0(m,n)}{3^m}\sspp \frac{1}{2^n\,3^m}},\  \dstyle{t\sspgr -\frac{c(m,n)\,x_0(m,n)}{2^n}} and  \dstyle{t\sspgr -\frac{c(m,n)\,y_0(m,n)}{3^m}}, \ie \pn
\dstyle{t\sspgeq \ceil{max\left( - \frac{c(m,n)\,x_0(m,n)}{2^n},\, -\frac{c(m,n)\,y_0(m,n)}{3^m}\right)}} \dstyle{\sspeq  \ceil{-c(m,n)\,min\left(\frac{x_0(m,n)}{2^n},\, \frac{y_0(m,n)}{3^m}  \right)}} \pn
\dstyle{\sspeq \ceil{-c(m,n)\, \frac{y_0(m,n)}{3^m} }\sspeq  t_{min}(m,n;+)}. \pn
If $c(m,n)\sspkl 0$ then  \dstyle{t\sspgeq \ceil{ \abs{c(m,n)}\, max\left(\frac{x_0(m,n)}{2^n},\, \frac{y_0(m,n)}{3^m}\right)}\sspeq \ceil{\abs{c(m,n)}\, \frac{x_0(m,n)}{2^n}}} $\sspeq  t_{min}(m,n;-)$. Thus with $t\sspeq t_{min}(m,n;sign(c)) \sspp k$, with $k\sspin \mathbb N_0$ one has the desired result. Note that $(x_0(m,n),\, y_0(m,n))$ is the smallest positive solution of the equation  $D(m,n;1)$,  eq.~\ref{Diophant}, because, for $c(m,n)\sspeq 1$, $t_{min}(m,n;+) \sspeq \ceil{-\frac{y_0(m,n)}{3^m}}$, but with $y_0(m,n)\sspin \{1,\,2,\,...\,3^m-1\}$ this is $0$.
\end{proof} \noindent
A proposition on the periodicity of the solution $y_0(m,n)$ follows.\psn
\begin{proposition} \label{y0Periodicity}
{\bf Periodicity of $\bf y_0(m,n)$ in $\bf n$}\psn
{\bf a)} The sequence $y_0(m,n)$ is periodic in $n$ with primitive period length $L_0\sspeq \varphi(3^m)$, for $m\sspin \mathbb N_0\, $ with {\sl Euler}'s totient function $\varphi(n)\sspeq$\seqnum{A000010}$(n)$, where $\varphi(1)\sspdef 1$. 
\psn
{\bf b)} The sequence $x_0(m,n\sspp L_0(m))\sspeq q(m)\,x_0(m,n)\sspm r(m)$, $m\sspin \mathbb N_0$,  with $q(m)\sspdef 2^{\varphi(3^m)}$ and \dstyle{r(m) \sspdef \frac{2^{\varphi{(3^m)}}\sspm 1}{3^m}}. See \seqnum{A152007}.\psn
{\bf c)} The set $Y_0(m)\sspdef \{y_0(m,n) \,|\, n\sspeq 0,\,1,\,...,\, \varphi(3^m)\sspm 1\} $, is, for $m\sspin \mathbb N_0$, a representation of the set $RRS(3^m)$, the smallest positive restricted residue system  modulo $3^m$. See \cite{TApostol} for the definition. The multiplicative group modulo $3^m$,  called \dstyle{\mathbb Z_{3^m}^{\times}\sspeq (\mathbb Z/3^m\,\mathbb Z)^\times} is congruent to the cyclic group $C_{\varphi(3^m)}$. See, e.g., \cite{WikiRRS},\psn
\end{proposition}
\begin{proof}\psn
{\bf a)} By {\sl Euler}'s theorem (\eg \cite{FR}, theorem 2.4.4.3 on p. 32) $a^{\varphi(n)} \sspequiv 1\, (mod\, n)$, provided $gcd(a,n)\sspeq 1$. Now \dstyle{gcd\left(\frac{3^m+1}{2},3^m\right) \sspeq gcd\left(\frac{3^m+1}{2},3\right)\sspeq 1} because \dstyle{\frac{3^m+1}{2} \sspequiv \frac{1}{2}\, (mod\, 3^m)} (see above) and hence  \dstyle{\frac{3^m+1}{2} \sspnotequiv 0\, (mod\, 3^m)}. This shows that $L_0(m)$ is a period length, but we have to show that it is in fact the length of the primitive period, \ie we have to prove that the order of  \dstyle{\frac{3^m+1}{2}} modulo $3^m$ is $L_0(m)$. (See \eg \cite{FR}, Definition 2.4.4.1. on p.31, for the order definition.) In other words we want to show that  \dstyle{\frac{3^m+1}{2}} is a primitive root (of $1$) modulo $3^m$. Assume that $k(m)$ is this order (the existence is certain due to {\sl Euler}'s theorem), hence $(\frac{1}{2})^{k(m)} \sspequiv 1\, (mod\, 3^m)$ and $k(m)\sspdiv L_0(m)$. It is known that the module $3^m$ possesses primitive roots, and the theorem on the primitive roots says that there are precisely $\varphi(\varphi(3^m))$ incongruent ones (\eg \cite{NZM}, pp. 205, 207, or \cite{Nagell}, theorem 62, 3., p. 104 and theorem, 65, p. 107). In our case this number is  $\varphi(2\cdot 3^{m-1}) \sspeq 2\cdot 3^{m-2}$ if $m \sspgeq 1$. The important point, proven in \cite{Nagell}, theorem 65.3 on p. 107, is that if we have a primitive root $r$ modulo an odd prime, here $3$, then, if $r^{3-1}\sspm 1$ is not divisible by $3^2$, it follows that $r$ is in fact a primitive root for any modulus $3^q$, with $q\sspin \mathbb N_0 $. One of the primitive roots modulo $3$ is $2$, because $2^2\sspeq 4\sspequiv 1\, (mod\, 3)$ and $2^1\sspnotequiv 1\, (mod\ 3)$. Also $2^{3-1}\sspm 1 \sspeq 3$ is not divisible by $3^2$, hence $2$ is a primitive root of any modulus  $3^q$ for $q\sspin \mathbb N_0$. From this we proof that \dstyle{\frac{3^m+1}{2}\sspequiv \frac{1}{2}\,(mod\, 3^m)} is a primitive root modulo $3^m$.  Consider  \dstyle{\left(\frac{3^m+1}{2}\right)^k\sspequiv \frac{1}{2^k}\,(mod\, 3^m)} for $k\sspeq 1,\,2,\,...,\, \varphi(3^m)$. In order to have \dstyle{\left(\frac{1}{2}\right )^k \sspequiv 1\, (mod\, 3^m)} one needs $2^k\sspequiv 1\, (mod\, 3^m)$. But due to \cite{Nagell} theorem 65.3. p. 107, for $p\sspeq 3$, a primitive root modulo $3^m$ is $2$, and the smallest positive $k$ is therefore $\varphi(3^m)$, hence \dstyle{\frac{3^m+1}{2}} is a primitive root (of $1$) of modulus $3^m$. \psn
{\bf b)} \dstyle{x_0(m,n\sspp \varphi(3^m))\sspeq  \frac{1\sspp 2^n\,2^{\varphi(3^m)}\,y_0(m,n)}{3^m}} from the periodicity of $y_0$. Rewritten as \dstyle{\frac{2^{\varphi(3^m)}\,\left( ( 2^{-\varphi(3^m)} \sspm 1) \sspp (1\sspp 2^n\,y_0(m,n) \right)} {3^m}\sspeq}  \dstyle{-\frac{1}{3^m}\,(2^{\varphi(3^m)}\sspm 1) \sspp 2^{\varphi(3^m)}\,x_0(m,n)\sspeq }  \dstyle{ q(m)\, x_0(m,n) \sspm r(m)} with the  values given in the {\sl Proposition}.\psn
{\bf c)} This follows from the reduced residue system modulo $3^m$ for  $m\sspin \mathbb N_0$, \pn
\dstyle{\left\{ \left(\frac{1}{2}\right)^0,\,\left(\frac{1}{2}\right)^1,\, ...,\, \left(\frac{1}{2}\right)^{\varphi(3^m)\sspm1} \right\}}, because \dstyle{\frac{1}{2}} is a primitive root modulo $3^m$ (from part {\bf b)}).  With $a(m)\sspdef \frac{3^m\sspp 1}{2}$  one has $1\sspeq gcd(a(m),3)\sspeq gcd(a(m),3^m)\sspeq gcd(a(m)^{b(m)},3^m)$ with $b(m)\sspdef 3^{m\sspm 1}$, also\pn
 \dstyle{\left\{a(m)^{b(m)}\, \left(\frac{1}{2}\right)^0,\,a(m)^{b(m)}\, \left(\frac{1}{2}\right)^1,\,...,\,a(m)^{b(m)}\, \left(\frac{1}{2}\right)^{\varphi(3^m)\sspm1}\right\}} is a reduced residue system modulo $3^m$ (see \cite{TApostol}, theorem 5.16, p. 113). Thus \pn
\dstyle{Y_0(m)\sspequiv \{a(m)^{b(m)}\,1,\, a(m)^{b(m)+1},\,...,\,  a(m)^{b(m) \sspp \varphi(3^m)\sspm 1}\}} is a reduced residue system modulo $3^m$. Therefore this gives a permutation of the reduced residue system modulo $3^m$ with the smallest positive integers sorted increasingly.
\end{proof}     \psn
\begin{example}  For $m\sspeq 3$, $\varphi(3^3)\sspeq 2\cdot 3^2\sspeq 18\sspeq L_0(3)$, \pn
$\{y_0(3,n)\}_{n=0}^{17}\sspeq \{26,\, 13,\, 20, \,10,\, 5,\, 16,\, 8,\, 4,\, 2, \,1,\, 14,\, 7,\, 17,\, 22,\, 11,\, 19,\, 23,\, 25\}$ a permutation of the standard reduced residue system modulo $27$, obtained by resorting the found system increasingly. See \seqnum{A239125}. For $m=1,\, 2$ and $ 4$ see \seqnum{A007583},\ \seqnum{A234038} and \seqnum{A239130}  for the solutions $(x_0(m,n),\, y_0(m,n))$. 
\end{example}
\psn
%%%%%%%%%%%%%%%%%%%%%%%%%%%%%%%%%%%%%%%%%
%%%%%%%%%%%%%%%%%%%%%%%%%%%%%%%%%%%%%%%%%
\section{Recurrences and their solution}\label{section4}
After these preparations it is straightforward to derive the recurrence for the start and end numbers $M$ and $N$ for any given $CW(\vec n_S)$, for $S\sspin \mathbb N$. \psn
{\bf A)} We first consider the case of words with $n_S\sspeq 1$. \ie $\vec n_S\sspeq [n_0,\,n_1\,,...\,n_{S-1},\,1]$. This is the word $CW(\vec n_S )\sspeq \rprod_{j=0}^{S-1}\, d^{n_j}\, s$ (with an ordered product, beginning with $j=0$ at the left-hand side). In order to simplify the notation we use $M(S)$, $N(S)$, $y_0(S)$, $x_0(S)$, and $c(S)$ for $ M(\vec n_S )$, $N(\vec n_S)$ ,  $y_0(S,n_S)$,  $x_0(S,n_S)$ and $c(S,n_S)$,  respectively. 
For $S=1$, the input for the recurrence, one has
\Beq
M(1;k) \sspeq 2^{n_0}\,(2\,k+1)\  \text{and}\ N(1;k) \sspeq 3\,k\sspp 2,\ {\text for}\ k \sspin \mathbb N\ ,
\Eeq 
because there are $n_0$ factors of $2$ from $d^{n_0}$, and then an odd number $2\,k\sspp 1$ leads after application of $s$ to $3\,k\sspp 2$. Thus $M(1)\sspeq 2^{n_0}$ and $N(1)\sspeq 2$.\psn
\begin{proposition} {\bf Recurrences for $\bf M(S)$ and $\bf N(S)$ with $\bf n_S\sspeq 1$} \label{Rec}\psn
{\bf a)} The coupled recurrences for $M(S,t)$ and $N(S,t)$, the first and last entry of the {\sl Collatz} sequences $CS(\vec n_S;t)$ for the word $ CW(\vec n_S)$ with $\vec n_S\sspeq [n_0,\,n_1,\,...\, n_{S-1},1]$  ($n_S\sspeq 1$) are \psn
\Bequarray
M(S,t)\sspeq M(S)&\sspp& 2^{\hat D(S)}\, t\ , \nonumber \\
N(S,t)\sspeq N(S)&\sspp& 3^S\,t\, ,\  {\rm with}\ t\sspin \mathbb Z\ , 
\Eequarray
 where \dstyle{\hat D(S)\sspdef \sum_{j=0}^{S-1}\,n_j} (we prefer to use a new symbol for the $n_S\sspeq 1$ case), and the recurrences for $M(S)$ and $\widetilde N(S)\sspeq N(S)\sspm 2$ are
 \Bequarray
M(S) \sspeq M(S-1) &\sspp& 2^{\hat D(S-1)}\,c(S-1)\, x_0(S-1) \ , \nonumber \\
 \widetilde N(S) &\sspeq& 3\,y_0(S-1)\, c(S-1)\, 
\Eequarray
with
\Beq
c(S-1) \sspeq 2\,(2^{n_{S-1}-2}\sspm 1) \sspm \widetilde N(S-1) \sspfed A(S-1) \sspm  \widetilde N(S-1)\ .
\Eeq
The recurrence for $c(S)$ is\psn
\Beq
c(S) \sspeq -3\,y_0(S-1)\,c(S-1) \sspp A(S)\  , S\sspgeq 2, 
\Eeq
and the input is $M(1)\sspeq 2^{n_0}$, $\widetilde N(1)\sspeq 0$ and $c(1) = A(1)$.\psn
{\bf b)} The general positive integer solution is
\Bequarray
M(S;k) &\sspeq & M(S) \sspp 2^{\hat D(S)}\,t_{min}(S-1) \sspp 2^{\hat D(S)}\,k\,,\nonumber \\
N(S;k) &\sspeq & 2\sspp \widetilde N(S) \sspp 3^S\,t_{min}(S-1) \sspp 3^S\,k, \ k\sspin \mathbb N_0 \, ,
\Eequarray 
 where 
\Beq
t_{min}(S) \sspeq t_{min}(S, n_S,sign(c(S))) \sspeq {\Caseszwei{\ceil{\abs{c(S)}\,\frac{x_0(S)}{2^{n_S}}}}{${\rm if} \ c(S)\sspkl 0 \ ,$} {\ceil{-c(S)\,\frac{y_0(S)}{3^S}}} {${\rm if} \ c(S)\sspgeq 0\ .$}}
\Eeq 
\end{proposition}
\begin{corollary}
\Bequarray
M(S;k) &\sspequiv& M(S) \sspp 2^{\hat D(S)}\,t_{min}(S-1)\, (mod\, 2^{\hat D(S)})\, , \nonumber \\
N(S;k) &\sspequiv& \widetilde N(S) \sspp  3^S\,t_{min}(S-1)\, (mod\, 3^S)\, .
\Eequarray
\end{corollary} \psn
In Terras' article \cite{Terras} the first congruence corresponds to {\sl theorem 1.2}, where the encoding vector $E_k(n)$ refers to the modified {\sl Collatz} tree using only $d$ and $s$ operations.
\begin{proof} \psn
{\bf a)} By induction over $S$. For $S\sspeq 1$  the input $M(1)\sspeq 2^{n_0}$, $N(1)\sspeq 2$ or $\widetilde N(1)\sspeq 0$ provides the start of the induction. Assume that part {\bf a)} of the proposition is true for $S$ values $1,\,2,\,...,\,S-1$. To find $M(S)$ one has to make sure that $d^{n_{S-1}}\,s$ can be applied to $N(S-1;k)$, the end number of step $S-1$ sequence $CS(\vec n_{S-1};t)$ which is $N_{int}(S-1,t)\sspeq N(S-1) +3^{S-1}\, t$, with integer $t$,  by the induction hypothesis. This number has to be of the form $2^{n_{S-1}-1}\,(2\, m\sspp1)$ (one has to have an odd number after $n_{S-1}$ $d-$steps such that $s$ can be applied). Thus \dstyle{3^{S-1}\,t \sspm 2^{n_{S-1}}\,m \sspeq 2^{n_{S-1}-1}\sspm N(S-1) \sspeq A(S-1)\sspm \widetilde N(S-1)\sspfed c(S-1)}, where $\widetilde N(S-1)\sspeq N(S-1) \sspm 2$ and $ A(S-1)\sspeq 2\,(2^{n_{S-1}-2}\sspm 1)$. Due to {\sl Lemma}~\ref{SolDiophant} the general solution, with $t\sspto x(S-1,n_{S-1};t)\sspentsp x(S-1;t)$, $m\sspto y(S-1,n_{S-1};t)\sspentsp y(S-1;t)$, to shorten the notation, is
\Bequarray
t\sspto x(S;t) &\sspeq& c(S-1)\,x_0(S-1) \sspp 2^{n_{S-1}}\,t \, , \nonumber \\ 
m\sspto  y(S;t) &\sspeq& c(S-1)\,y_0(S-1) \sspp 3^{S-1}\,t\, , \ t\sspin \mathbb Z\ . 
\Eequarray   
Therefore the first entry of the sequence $CS(\vec n_S;t)$ is \dstyle{M(S;t)\sspeq M(s-1,x(S-1,t))} which is
\Beq
M_{int}(S;t)\sspeq M(S-1) \sspp 2^{\hat D(S-1)}\,c(S-1)\,x_0(S-1) \sspp 2^{\hat D(S)}\, t\, ,
\Eeq
hence \dstyle{M(S)\sspeq M(S-1)\sspp 2^{\hat D(S-1)}\,c(S-1)\, x_0(S-1)}, the claimed  recurrence for $M(S)$.\pn
The last member of $CS(\vec n_{S-1};t)$ is $3\,m+2$ (after applying $s$ on $2\,m\sspp 1$ from above). Thus \dstyle{N_{int}(S;t) \sspeq 3\,y(S;t)\sspp 2}, or \dstyle{N_{int}(S;t)\sspm 2\sspeq 3\,c(S-1)\,y_0(S-1)\sspp 3^S\, t}. Therefore, $\widetilde N(S) \sspeq N(S)\sspm 2 \sspeq 3\,c(S-1)\,y_0(S-1)$  the claim for the $\widetilde N$ recurrence. Note that the remainder structure of eqs. $(20)$ and $(21)$, expressed also in the {\sl Corollary}, has also been verified by this inductive proof. The recurrence for $c(S)\sspeq A(S) \sspm \widetilde N(S)$  follows from the one for $\widetilde N(S)$. \psn
{\bf b)} Positive integer solutions from $M_{int}(S;t)$ and $N_{int}(S;t)$ of part {\bf a)} are found from the second part of  {\sl Lemma} ~\ref{SolDiophant} applied to the equation $3^{S-1}\, x \sspm 2^{n_{S-1}}\,y\sspeq c(S-1)$, determining $t_{min}(S-1)$ as claimed. This leads finally to the formulae for $M(S;k)$ and $N(S;k)$ with $k\sspin \mathbb N_0$.  
\end{proof}
\begin{example} {\bf $\bf (sd)^{S-1}\,s$ Collatz sequences}\psn
Here $n_0 \sspeq 0 \sspeq  n_S $ and $n_{j}\sspeq 2$ for $j\sspeq 1,\,2\, ...,\,S-1$.
The first entries $M(S;k)$ and the last entries $N(S;k)$ of the {\sl Collatz} sequence $CS([0,2,...,2];k)$  (with $S-1$ times a $2$), whose length is $3\,S$,  are $M(S;k)\sspeq 1\sspp 2^{2\,S-1}\,k$ and $N(S;k)\sspeq 2 \sspp 3^S\,k$. For $S=3$ a complete {\sl Collatz} sequence $CS([0,2,2];3)$ of length $9$ is $[97, 292, 146, 73, 220, 110, 55, 166, 83]$ which is a special realization of the word $sdsds$ with start number $M(3;3)\sspeq 97$ ending in $N(3;3)\sspeq 83$. Note that for this $u-d$ pattern the start and end numbers have remainders $ M(S;0) = M(1;0) \sspeq 1$  and $N(S;0)\sspeq N(1;0) \sspeq 2$. See the tables \seqnum{A240222} and \seqnum{A240223}.
\end{example}\psn
The recurrences for $M(S)\sspentsp M(\vec n_{S-1})$,  $\tilde N(S)\sspentsp \tilde N(\vec n_{S-1})$ or $N(S)\sspentsp N(\vec n_{S-1})$ and $c(S)\sspentsp c(\vec n_{S-1}$ are solved by iteration with the given inputs $M(1)\sspeq 2^{n_0}$, $\tilde N(1)\sspeq 0$ and $c(1)\sspeq A(1)\sspeq 2\,(2^{n_0-2}\sspm 1)$. \psn
\begin{proposition} {\bf Solution of the recurrences for $\bf n_s\sspeq 1$} \label{SolRec}\psn
The solution of the recurrences of {\sl Proposition}~\ref{Rec} with the given inputs are, for $S\sspin \mathbb N$:\psn
\Bequarray
c(S)&\sspeq&  A(S) \sspp \sum_{j=1}^{S-1} \,(-3)^j\,A(S-j)\,\prod_{l=1}^j\,y_0(S-l) \, , \nonumber \\
\tilde n(S)&\sspeq& A(S) \sspm  c(S) \sspeq -\sum_{j=1}^{S-1}\,(-1)^j\,A(S-j)\, \prod_{l=1}^{j}\,y_0(S-l)\, , \nonumber \\
N(s)&\sspeq& \tilde N(S)\sspp 2\, , \nonumber \\
M(S) &\sspeq& 2^{n_0}\sspp \sum_{j=1}^{S-1}\,R(S-j)\, , 
\Eequarray
with \dstyle{\hat D(S)\sspdef 1\sspp \sum_{j=0}^{S-1}\, n_j},\  \dstyle{A(S)\sspdef 2\,(2^{n_S-2}\sspm 1)},\ \dstyle{R(S)\sspdef 2^{\hat D(S)} \,x_0(S)\,c(S)}\  and $y_0(S)\sspentsp y_0(S,n_S)$, $x_0(S)\sspentsp x_0(S,n_S)$, given in {\sl Lemma} ~\ref{SolDiophant}.
\end{proposition}
\begin{proof}\psn
This is obvious.
\end {proof} \pn
{\bf B)} The general case $n_S \sspgeq 1$ can now be found by appending the operation $d^{n_S-1}$ to the above result. This leads to the following {\sl theorem}.\psn
\begin{theorem} \label{GenSol} {\bf The general case $\bf {\overrightarrow n}_S$} \psn
 For the Collatz word $CW(\vec n_S)\sspeq d^{n_0}\,\rprod_{j=1}^{S}\,(s\,d^{n_j-1}) \sspeq d^{n_0}\, s\,\rprod_{j=1}^{S}\,(d^{n_j-1}\, s)\,d^{n_S-1}$ (the ordered product begins with $j=1$ on the left-hand side) with  $n_0\sspin \mathbb N_0$ , $n\sspin \mathbb N$, the first and last entries of the corresponding Collatz sequences $\{CS(\vec n_S;k)\}$, of length $L(S) \sspeq n_0 \sspp 2\,\sum_{j=1}^S n_j$,  for $k\sspin \mathbb N_0$, are\psn
\Bequarray
M(\vec n_S;k) &\sspeq& M(S) \sspm 2^{\hat D(S)}\,N(S)\,x_0(S,n_S-1)\sspp 2^{D(S)}\,t_{min}(S,n_S-1,sign(c_{new}(S)) \nonumber\\
&&  \sspp 2^{D(S)}\,k\, , \nonumber \\
N(\vec n_S;k) &\sspeq& c_{new}(S)\,y_0(S,n_S-1)\sspp 3^S\, t_{min}(S,n_S-1,sign(c_{new}(S))\sspp3^S\, k\ ,
\Eequarray
\end{theorem}\noindent
with $c_{new}(S)\sspdef -N(S)$, $\hat D(S) \sspeq 1\sspp \sum_{j=0}^{S-1}\,n_j$, $D(S)\sspeq \sum_{j=0}^S\, n_j$.
\begin {proof}\psn
In order to be able to apply to the {\sl Collatz} sequences $CS([n_0,n_1,...,n_S-1,1])$ (with the results from part A) above) the final $d^{n_S-1}$ operation one needs for the last entries  $N_{int}(S;t) \sspeq N(S) \sspp 3^S\,t \sspeq 2^{n_S-1}\,m$ with some (even or odd) integer m. The new last entries of $CS([n_0,n_1,...,n_S];t)$ will then be $m$. The general solution of $3^S\sspm 2^{n_S-1}\,m\sspeq -N(S)\sspfed c_{new}(S)$ is according to {\sl Lemma}~\ref{SolDiophant} \psn
\Bequarray
t\sspto x(S;t)&\sspeq&  c_{new}(S)\, x_0(S,n_S-1)\sspp 2^{n_S-1}\,t,\, \nonumber \\
m\sspto y(S;t)&\sspeq&  c_{new}(S)\, y_0(S,n_S-1)\sspp 3^S\,t,\,\ t\sspin \mathbb Z\, . 
\Eequarray
This leads to positive integer solutions after the shift $t\sspto t_{min}\sspp k$,  with \pn
$t_{min} \sspeq t_{min}(S,n_S-1,sign(c_{new}(S)))$ to the claimed result $N(S;k)$ for the new last number of $CS(\vec n;k)$, with $k\sspin \mathbb N_0$. The new start value $M(S;k)$ is obtained by replacing $t\sspto x(S;t)$ in the old $M_{int}(S;t)$ (with $n_S\sspeq 1$). $M(S;k)\sspeq M_{int}(S,x(S;t))$ with $t\sspto t_{min}\sspp k$, also leading to the claimed formula.
\end{proof} \psn
The remainder structure modulo $2^{D(S)}$ for $M(\vec n_S,k)$ and modulo $3^S$ for $N(\vec n_S,k)$ is manifest.\psn
The explicit sum versions of the results for case $n_S\sspeq 1$, given in {\sl Proposition}~\ref {GenSol}, can be inserted here.
\begin{example} {$\bf ud^{m}\sspeq sd^{m-1}$}\psn
For $m \sspeq  1,\,2,\, 3$ and $ k\sspeq 0,\,1,\,...,\,10$ one finds for $N([0,m],k)$:\pn    $[2, 5, 8, 11, 14, 17, 20, 23, 26, 29, 32]$,\  $[1, 4, 7, 10, 13, 16, 19, 22, 25, 28, 31]$,\ \pn
$ [2, 5, 8, 11, 14, 17, 20, 23, 26, 29, 32]$, and for $M([0,m],k)$:\pn 
 $[1,3,5,7,9,11,13,15,17,19,21]$,\ $[1,5,9,13,17,21,25,29,33,37,41]$,\ \pn
$[5,13,21,29,37,45,53,61,69,77,85]$. Only the odd members of $N([0,m],k)$, that is the odd indexed entries, and the corresponding $M([0,m],k)$  appear in \cite{Truemper}, example 2.1. See \seqnum{A238475} for $M([0,2\,n],k)$  and  \seqnum{A238476} for  $M([0,2\,n\sspm 1],k)$. The odd $N([0,2\,n],k)$ values are the same for all $n$, namely $5\sspp 6\,k$, and  $N([0,2\,n-1],k)\sspeq  1\sspp 6\,k$ for all $n\sspin \mathbb N$.
\end{example}  
\begin{example} {$\bf (ud)^{n}\sspeq s^{S}, S\sspin \mathbb N$}\psn
$\vec n_S\sspeq [0,1,...,1]$ with $S$ times a $1$. For $S \sspeq  1,\,2,\, 3$ and $ k\sspeq 0,\,1,\,...,\,10$ one finds $N(\vec n_S,k)$\ \pn
 $[5,\,8,\,11,\,14,\,17,\,20,\,23,\,26,\,29,\,32]$,\ $[17,\,26,\,35,\,44,\,53,\,62,\,71,\,80,\,89,\,98]$,\ \pn
$[53,\,80,\,107,\,134,\,161,\,188,\,215,\,242,\,269,\,296]$, and for $M(\vec n_S,k)$:\pn
 $[3,\, 5,\, 7, \,9,\, 11,\, 13,\, 15,\, 17,\, 19, \,21]$,\ $ [7,\, 11,\, 15,\, 19, \,23,\, 27, \,31, \,35,\, 39,\, 43]$,\ \pn
$[15, \,23,\, 31, \,39, \,47, \,55, \,63, \,71, \,79, \,87]$,. For odd $N$ entries, and corresponding $M$ entries this is  \cite{Truemper}, example 2.1. See \seqnum{A239126} for these $M$ values, and \seqnum{A239127} for these $N$ values, which are here $S$ dependent.
\end{example} \psn
In conclusion the author does not think that the knowledge of all {\sl Collatz} sequences with a given up-down pattern (a given {\sl Collatz} word) will help to prove the {\sl Collatz} conjecture. Nevertheless the problem considered in this paper is a nice application of a simple {\sl Diophan}tine equation.\pbn
{\bf Acknowlegement}\psn
Thanks go to {\sl Peter Bala} who answered the author's question for a proof that all triangle \seqnum{A107711} entries are non-negative integers. See the history there,  Feb 28 2014.
\pbn
\pbn

\pbn 
\pbn
\hrule \psn
2010 Mathematics Subject Classification: Primary 11D04; Secondary 32H50.
\psn 
{\em Keywords:}  Collatz problem, Collatz sequences, Collatz tree, recurrence, iteration, linear Diophantine equation. \psn
\hrule\psn 
(Concerned with sequences \seqnum{A000010},\ \seqnum{A005186},\ \seqnum{A007583}, \ \seqnum{A016957},\ \seqnum{A107711}, \ \seqnum{A127824},\ \seqnum{A152007},\ \seqnum{A176866},\ \seqnum{A234038}\ \seqnum{A238475},\  \seqnum{A238476},\ \seqnum{A239125},\ \seqnum{A239126},\ \seqnum{A239127},\  \seqnum{A239130},\ \seqnum{A240222},\ \seqnum{A240223}.)


\begin{thebibliography}{9}
 \bibliographystyle{plain}
\bibitem{TApostol} T. M. Apostol, Introduction to Analytic Number Theory, Springer-Verlag, 1976, page 113. 
\bibitem{FR} B. Fine and G. Rosenberger, Number Theory, Birkh\"auser, 2007.
\bibitem{Lagarias}  J. Lagarias, The $3x+1$ Problem and its Generalization, Amer. Math. Monthly, 92(1) (1985) 3-23.
\bibitem{Nagell} T. Nagell, Introduction to Number Theory, Chelsea Publishing company, New York, 1964.
\bibitem{NZM}  I. Niven, H. S. Zuckerman and H. L. Montgomery, An Introduction to the Theory Of Numbers, Fifth Edition, John Wiley and Sons, Inc., NY 1991, pp. 212-214
\bibitem{OEIS} The On-Line Encyclopedia of Integer Sequences (2010), published electronically at \url{http://oeis.org}.
\bibitem{Terras} R. Terras, A stopping time problem on the positive integers, Acta Arith. 30 (1976) 241-252.  
\bibitem{Truemper} M. Tr\"umper, The Collatz Problem in the Light of an Infinite Free Semigroup , Chinese Journal of Mathematics, November 2013,\pn
 \url{http://www.hindawi.com/journals/cjm/aip/756917/}. 
\bibitem{WeissteinC} Eric Weisstein's World of Mathematics, Collatz Problem, \pn
\url{http://mathworld.wolfram.com/CollatzProblem.html}. 
\bibitem{WikiC}Wikipedia, Collatz conjecture,\pn
 \url{https://en.wikipedia.org/wiki/Collatz_conjecture}.
\bibitem{WikiRRS} Wikipedia, Multiplicative group of integers modulo n, \pn
\url{https://en.wikipedia.org/wiki/Multiplicative_group_of_integers_modulo_n}.
 \end{thebibliography}
\end{document}